\RequirePackage{fix-cm}
\documentclass[smallextended,referee]{svjour3}       
\smartqed  
\usepackage{graphicx}
\usepackage{dsfont}
\usepackage{nicefrac}
\usepackage{enumerate}
\usepackage{amsfonts}

\usepackage{amsmath}
 \usepackage{mathptmx}      
\usepackage{latexsym}
\newcommand{\1}{\mathds{1}}

\newcommand{\EE}{\mathds{E}}
\newcommand{\PP}{\mathds{P}}
\newcommand{\e}{\textrm{e}}
\newcommand{\di}{\textrm{d}}

\spnewtheorem*{assp}{Assumptions on the local time}{\it}{\rm}

\begin{document}

\title{Universality for persistence exponents of local times of self-similar processes with stationary increments \thanks{The author's research is funded by the Deutsche Forschungsgemeinschaft (DFG, German Research Foundation) through the priority programme SPP 2265 ``Random Geometric Structures'', grant no.\ 443916008.}
}
\subtitle{}

\titlerunning{Persistence of H-sssi local times}        

\author{Christian M\"onch}


\institute{C. M\"onch \at
              Institut für Mathematik\\
              Johannes Gutenberg-Universität\\
              Staudingerweg 9,         
              D-55099 Mainz\\
              \email{cmoench@uni-mainz.de}           
}

\date{Received: date / Accepted: date}

\maketitle

\begin{abstract}
We show that $\PP ( \ell_X(0,T] \leq 1)=(c_X+o(1))T^{-(1-H)}$, where $\ell_X$ is the local time measure at $0$ of any recurrent $H$-self-similar real-valued process $X$ with stationary increments that admits a sufficiently regular local time and $c_X$ is some constant depending only on $X$. A special case is the Gaussian setting, i.e. when the underlying process is fractional Brownian motion, in which our result settles a conjecture by Molchan [Commun. Math. Phys. 205, 97-111 (1999)] who obtained the upper bound $1-H$ on the decay exponent of $\PP ( \ell_X(0,T] \leq 1)$. 
Our approach establishes a new connection between persistence probabilities and Palm theory for self-similar random measures, thereby providing a general framework which extends far beyond the Gaussian case.
\end{abstract}

\section{Introduction: Persistence probabilities for fractional Brownian motion}
\label{intro}

We study local times of stochastic processes from the point of view of {persistence probabilities}, i.e. the probabilities that a stochastic process remains inside a relatively small subset of its state space for a long time. The problem of calculating persistence probabilities is a very active field of mathematical research, see e.g. the recent articles \cite{aurzadabuck2018persistence,aurzada2018persistence,lyu2018persistence,molchan2018persistence,poplavskyi2018exact,amzMarkov}, an overview of the developments over the last decades is given by Aurzada and Simon in \cite{aurzada2015persistence}. The main motivation to study the persistence behaviour of stochastic systems is its great relevance for certain areas of statistical physics, see e.g. \cite{caravenna2008pinning,constantin2004persistence,majumdar1999persistence} and the survey by Bray et al. \cite{bray2013persistence}.

The starting point of our investigation is Molchan's celebrated and by now classical results \cite{molchan1999maximum} concerning the persistence of fractional Brownian motion on the real line, which we briefly summarise now. Let $B=(B_t)_{t\in\mathds{R}}$ denote a $1$-dimensional fractional Brownian motion (FBM) of Hurst index $H\in(0,1)$. $B$ can be characterised as the unique (up to multiplication by a constant) Gaussian process which is \emph{$H$-self-similar} with \emph{stationary increments ($H$-sssi)}. In \cite{molchan1999maximum} it is shown that the maximum process $\bar{B}_t=\max_{0\leq s \leq t}B_s, t\geq 0,$ of $B$ satisfies
\begin{equation}\label{eq:maxfbm}
\PP(\bar{B}_T \leq 1)= T^{-(1-H)+o(1)}.
\end{equation}
Subsequently, improved bounds on the error estimate implicit in \eqref{eq:maxfbm} have been derived by Aurzada \cite{aurzada2011one} and by Aurzada et al. \cite{aurzada2018persistence}. Note that, using self-similarity, we may replace the boundary $1$ in \eqref{eq:maxfbm} by any fixed value $x>0$ without changing the order of decay. The probability in \eqref{eq:maxfbm} is called the \emph{persistence probability} and the corresponding exponent $\bar{\kappa}=1-H$ the \emph{persistence exponent} of $B$. In \cite{molchan1999maximum}, Molchan also considered the lower tail probabilities of several other path functionals, namely
\begin{itemize}
	\item[$\bullet$] $\ell(0,T]:=\lim_{\epsilon\to 0}(2\epsilon)^{-1}\int_0^T \mathds{1}\{B_t\in (-\epsilon,\epsilon)\}\di t$, the local time at $0$,
	\item[$\bullet$] $\tau^{+}:= \inf\{t\geq 1: B_t=0\}$, the first zero after time $1$,
	\item[$\bullet$] $\sigma^{+}_T:= \int_0^T \mathds{1}\{B_t>0\}\di t$, the time spent in the positive half-axis,
	\item[$\bullet$] $\tau^{\text{max}}_T:= \arg \max \{B_t, t\in[0,T]\}$, the time at which the maximum is achieved;
\end{itemize}
and his results imply that \begin{equation}\label{eq:otherexp}
\lim_{T\to\infty}\frac{\log \PP(\tau^{+}\geq T)}{\log T}=\lim_{T\to\infty}\frac{\log \PP(\sigma^{+}_T\leq 1)}{\log T}
= \lim_{T\to\infty}\frac{\log \PP(\tau^{\text{max}}_T\leq 1)}{\log T}= H-1.
\end{equation}
These asymptotics can be viewed as a general (and very weak) form of L{\'e}vy's arcsine-laws for Brownian motion. Intuitively, the agreement of exponents can be explained by observing that the dominating events contributing to each of the probabilities in \eqref{eq:maxfbm}
and \eqref{eq:otherexp} are long (negative) excursions of $B$ from $B_0=0$. This type of event also entails a small local time at $0$, and one is inclined to believe that the probability of the local time being small is of the same order. However, the result in \cite{molchan1999maximum} for the local time is only a lower bound, namely that there is a constant $b\in(0,\infty)$ such that \begin{equation}\label{eq:UB}
\PP(\ell(0,T]\leq 1) \geq T^{-(1-H)}\,b \e^{-\sqrt{\log T}},
\end{equation}
for sufficiently large $T$. Hence, the \emph{local time persistence exponent} 
\[
\kappa = -\lim_{T\to\infty}\frac{\PP(\ell(0,T]\leq 1)}{\log T}
\] of $B$ satisfies 
\[
\kappa\leq 1-H,
\]
and the upper bound given in \eqref{eq:UB} is still the best lower tail estimate\footnote{The Markovian case $H=\nicefrac12$, is of course exceptional. The exact distribution of $\ell(0,T]$ for Brownian motion was already obtained much earlier by L{\'e}vy \cite{levyprocessus}.} for the local time of FBM with index $H\in(0,1)\setminus\{\nicefrac{1}{2}\}$ available in the literature.

Molchan's proofs require some technical tools, namely Slepian's Lemma and reproducing kernel Hilbert spaces, which are specific to the Gaussian setting. Moreover, his argument crucially relies on the connection of the persistence probability to a certain path integral functional and this relation is in fact also useful outside the FBM context, see e.g.\ \cite{agpp2015persistence}. However, there is no straightforward link between the functional and the properties of $\ell$ which is the reason why Molchan's result regarding the local times is incomplete.

The goal of the present paper is to show how to circumvent these obstacles and establish the equality \begin{equation}
\label{eq:kappakappa}
\kappa = 1-H
\end{equation}
for FBM directly by studying the local time. In fact, we prove a significantly stronger result, namely that there is a constant $C\in(0,\infty)$, such that
\begin{equation}\label{eq:newbounds}
\PP(\ell((0,T]\leq 1)\sim C T^{-(1-H)},
\end{equation}
where here and in what follows we use the notation $f(T)\sim g(T)$ to indicate that the ratio of the two functions $f,g$ converges to $1$ as the argument $T$ approaches $\infty$. Our approach to show \eqref{eq:newbounds} does not use that FBM is a Gaussian process. Consequently, \eqref{eq:newbounds} not only holds for FBM, but for \textbf{any} $H$-sssi process which admits sufficiently regular local time measures.

In fact, \eqref{eq:kappakappa} relates the time $B$ spends at $0$ to the box-counting dimension\footnote{Heuristic scaling arguments often use the box-counting dimension to capture the fractality of a set due to its rather intuitive definition, whereas the Hausdorff dimension is generally preferrable from a mathematical point of view, see e.g. \cite{falconer2004fractal} for a discussion. Both notions of fractal dimension coincide for many random fractals and in particular for the zero set of FBM. In fact, it was recently shown by Mukeru \cite{Mukeru2018} that the level sets of FBM even have Fourier dimension $1-H$.} of its zero set, which equals $1-H$. Indeed, there is a well known non-rigorous box-counting argument, see e.g. \cite{ding1995distribution}, which suggests that the probability of observing an excursion from $0$ of length greater than $T$ is of order $T^{-{(1-H)}}$. One way of looking at our result is that it makes this connection rigorous; our method indeed enables us to prove \eqref{eq:kappakappa} using only the invariance properties of the underlying processes, without recurrence to specific distributional structures such as Gaussianity, the Markov or Martingale property, etc.

It is immediate from \eqref{eq:kappakappa} that we have $\bar{\kappa}=\kappa$ for $B$ 
and the author believes that this is also true in a more general context. In particular, one should be able to combine the arguments in this paper with the methods developed by Aurzada et al. in \cite{aurzada2018persistence,aurzada2016persistence} to show that both persistence exponents coincide for all $H$-sssi processes which are positively associated. 

The technique for establishing $\kappa=1-H$ proposed below is entirely novel in the context of persistence probabilities. It combines three principal ingredients:
\begin{itemize}
	\item[$\bullet$]  a distributional representation of the local times using Palm theory;
	\item[$\bullet$]  an invariance property of Palm distributions known as \emph{mass stationarity};
	\item[$\bullet$]  a simple bi-variate scaling relation for an associated point process, which encodes the $H$-sssi property.
\end{itemize}
The first two items are covered by a very general and deep result of Last, M\"orters and Thorisson \cite{LMT14}, c.f.\ Section \ref{sec:auxiliaryproof} below, but only the first item actually needs abstract results from the theory of random measures which are not based on elementary calculations.

The remainder of this text is organised as follows. In the next section we fix our notation and present our main result for the local time persistence probabilities in a general setting, Theorem~\ref{thm:persistence1} and in two special cases, namely for FBM and the Rosenblatt process. The main argument to prove Theorem~\ref{thm:persistence1} is given in Sections~\ref{sec:proofs}, subject to some auxiliary results which require a more extensive discussion in Section \ref{sec:auxiliaryproof}. 
In the appendix we provide some useful results from the literature as well as some auxiliary calculations for convenience of the reader.

\section{Notation and main results}\label{sec:results1}
We assume throughout the remainder of the article that $(X_t)_{t\in\mathds{R}}$ is a real-valued stochastic process defined on a complete probability space $(\Omega,\mathcal{F},\PP)$, which satisfies $X_0=0$ almost surely, and is {continuous in probability at $0$}, i.e. \begin{equation}\label{eq:contprob}
\lim_{h\to 0}\PP(|X_{h}|>\epsilon)=0.
\end{equation}
Most importantly, we require that $X$ to be \emph{$H$-sssi}, i.e.\ to satisfy the invariance relations
\[
(X_{t+s}-X_t)_{s\in\mathds{R}}\overset{d}{=}(X_{s})_{s\in\mathds{R}}, \;\text{ for any }t\in\mathds{R},\quad \quad \text{(stationarity of increments),}
\]
and
\[
(X_{rs})_{s\in\mathds{R}}\overset{d}{=}(r^H X_s)_{s\in\mathds{R}}, \;\text{ for every } r>0, \quad\quad \text{(}H\text{-self-similarity)},
\]
where $\overset{d}{=}$ denotes equality of finite dimensional distributions. We extend the definition of $H$-sssi to processes indexed by $[0,\infty)$ by restricting the stationarity of increments to positive shifts only. 
\begin{remark}
It is not hard to see, c.f.\ \cite[Lemma 1.1.1]{embrechts2009selfsimilar}, that self-similarity and \eqref{eq:contprob} in fact imply $\PP(X_0=0)=1$. Note further, that stationarity of increments allows us to conclude continuity in probability at all times from \eqref{eq:contprob}.
\end{remark}

Our main object of interest is the \emph{local time measure} of $X$ at $0$. We use the following notational conventions related to measures: $\mathfrak{B}(\cdot)$ denotes the Borel-$\sigma$-field of the space in brackets. If $\nu$ is a measure on $\mathfrak{B}(\mathds{R})$, the Borel-sets of the real line, and $(a,b)$ is an interval, we use the notation $\nu(a,b)$ instead of $\nu((a,b))$ and an analogous shorthand for closed and half open intervals. We frequently associate a measure $\nu$ with its \emph{additive functional} \[
\nu_t=\begin{cases}
\nu(0,t], &\text{ if }t>0,\\
-\nu(-t,0], &\text{ if }t\leq 0,
\end{cases} \]	
and vice versa. If $\nu$ is a random measure, then $(\nu_t)_{t\in\mathds{R}}$ is a non-decreasing stochastic process. We define the occupation measure of $X$ on $\mathfrak{B}(\mathds{R}\times \mathds{R})$ by setting
\begin{equation}\label{eq:occup}
\psi(A \times B)=\int_A \mathds{1}\{X_r \in B\} \di r,\quad A,B\in\mathfrak{B}(\mathds{R}),
\end{equation}
and recall that \eqref{eq:occup} yields a well-defined Borel measure as long as the trajectories of $X$ are Borel-functions. We say that $X$ {admits local times}, or shorter $X$ is \textsf{LT},  if for each $n=1,2,\dots$, $\PP$-a.s.,
\[
\psi\big((-n,n)\times \cdot\big) \text{ is absolutely continuous w.r.t. Lebesgue measure}.
\]
Since $X$ has stationary increments it is in fact sufficient for $X$ to be \textsf{LT}, that the Radon-Nikodym density
$
\nicefrac{\di \psi(I,\di y)}{\di y}
$ exists a.s. for an arbitrary open set $I$. Disintegration yields, for every $y$ outside some Lebesgue-negligible set $\mathcal{R}$, a locally finite measure $\ell^y$ on $\mathfrak{B}(\mathds{R})$ such that
\begin{equation}\label{eq:loctimedef}
\psi(A \times B)=\int_B \ell^y(A)\di y,\quad A,B\in\mathfrak{B}(\mathds{R}),
\end{equation}
and we call $\ell^y$ the local time of $X$ at level $y$. Moreover, it can be shown, see e.g. \cite[Lemma (3)]{geman1974local}, that for every $y\notin\mathcal{R}$, we can choose a version of $\ell^{y}(0,t]$ which is right-continuous in the time variable. A similar statement holds for $\ell^{y}(-t,0]$. Recall that we have $X(0)=0$ a.s., i.e. $\ell^0=\ell^{X(0)}$, if $0\notin\mathcal{R}$, but a priori the existence of $\ell^0$ cannot be guaranteed using the above construction of local times. This technical issue is addressed in the Appendix, for the time being let us assume that $0\notin \mathcal{R}$ and that $\ell^0$ is well-defined.

We are chiefly interested in $\ell^0$ and therefore just abbreviate $\ell=\ell^0$ and call it local time, without reference to the level $0$. From the construction of $\ell$ we can straightforwardly derive a path-wise representation. Let \[
\ell_\epsilon^y(A):=\frac{1}{2\epsilon}\psi\left(A\times (y-\epsilon,y+\epsilon)\right),
\] 
then we have that for all $y\notin\mathcal{R}$
\begin{equation}\label{eq:LTpathwise}
\lim_{\epsilon\to 0}\ell_\epsilon^{y}(A)= \ell^{y}(A),\; A\in\mathfrak{B}(\mathds{R}),
\end{equation}
which shows that our definition agrees with the formula for $\ell$ given in the introduction.\\ 

We now introduce two further structural conditions on $\ell$ which are necessary for our derivation of the lower tail probabilities of $\ell(0,T]$. Let \[\mathsf{supp}(\nu)=\{t: \nu(t-\epsilon,t+\epsilon)>0 \text{ for all }\epsilon>0 \}\] denote the support of a measure $\nu$ on $\mathfrak{B}(\mathds{R})$ and recall that a nowhere dense set of real numbers is a set whoose (topological) closure does not contain any interval.
\begin{assp} With probability $1$,
	\begin{equation}\label{eq:continuity}\tag*{\textsf{AL}}
	\ell \text{ has no atoms},
	\end{equation}
	\begin{equation}\label{eq:singularity}\tag*{\textsf{ND}}
	\mathsf{supp}(\ell) \text{ is nowhere dense.}
	\end{equation}
\end{assp}
Condition \ref{eq:continuity} is equivalent to demanding that $(\ell_t)_{t\in\mathds{R}}$ be continuous a.s. Both conditions entail a rather erratic behaviour of the trajectories of $X$, which is not surprising in view of our main example, fractional Brownian motion. 
\begin{remark}
The validity of \ref{eq:continuity} and \ref{eq:singularity} are indispensable for the approach to local times taken in this paper. However, the author strongly believes that the conditions listed are not minimal and that in particular condition~\ref{eq:singularity} is a consequence of condition~\ref{eq:continuity} for any $H$-sssi process which is continuous in probability, but is not aware of any proof of this implication.
\end{remark}
 To exclude pathologies, we also restrict ourselves to situations where $\ell$ is a.s. not the zero measure, we then say $\ell$ is \emph{non-zero}.
We are now prepared to state our main result.
\begin{theorem}[Persistence of local time for $H$-sssi processes]\label{thm:persistence1}
	Let $X$ be continuous in probability, $H$-sssi and \textsf{LT} and denote by $\ell$ its local time at $0$. If $\ell$ is non-zero and satisfies \ref{eq:continuity} and \ref{eq:singularity}, then there exists a constant $c_X\in(0,\infty)$ such that, as $T\to\infty$,
	\[
	\PP(\ell(0,T]\leq 1) \sim c_X T^{-(1-H)}.
	\]
\end{theorem}
There are two important observations needed for the proof of Theorem~\ref{thm:persistence1}. The first one is a result stating that the lengths of the excursions of $X$ from $0$ follow, in a certain sense, a hyperbolic distribution. The precise formulation is given in Section~\ref{sec:proofs} as Proposition~\ref{prop:powerlaw}. The second one is that, under \ref{eq:continuity} and \ref{eq:singularity}, $\ell$ is entirely encoded in the excursions of $X$ from $0$, which is manifested in the fact that the right-continuous inverse of $(\ell_t)_{t\in \mathds{R}}$ is a.s. a pure jump process. Before we develop the details we devote the remainder of this section to some of the implications of Theorem~\ref{thm:persistence1}.
To this end we provide two example processes, for which Theorem~\ref{thm:persistence1} can be applied. The first one, naturally, is fractional Brownian motion. The local times and level sets of FBM have been studied by several authors, the pioneering work was done by Kahane in the late 1960's, see in particular \cite[Chapter 18]{kahane1985some}.
\begin{theorem}[Persistence of local time for FBM]\label{thm:persistence2}
	Let $B$ denote fractional Brownian motion with Hurst index $H\in(0,1)$ and denote by $\ell$ its local time at $0$.  Then there is a constant $c_B\in(0,\infty)$ such that, as $T\to\infty$,
	\[
	\PP(\ell(0,T]\leq 1)\sim c_B T^{-(1-H)}.
	\]
\end{theorem}
\begin{proof}
	We only need to verify conditions \ref{eq:continuity} and \ref{eq:singularity}. The continuity in time of fractional Brownian local time is well known, e.g. by applying the criterion proposed by Geman \cite{geman1976note} for Gaussian processes. Let $\mathcal{Z}_B$ denote the set of zeroes of the FBM trajectory. Kahane \cite{kahane1985some} showed that the Hausdorff dimension $\dim \mathcal{Z}_B$ of the zero set equals $1-H<1$ a.s. Since $B$ is a.s. continuous, it follows that $\mathcal{Z}_B$ is closed. Together with its non-integer dimension this implies that $\mathcal{Z}_B$ is nowhere dense and therefore $\mathsf{supp}(\ell)$ is nowhere dense, since $\mathsf{supp}(\ell)\subset \mathcal{Z}_B$.
\end{proof}
To illustrate the power of our approach, we now discuss a non-Gaussian example, namely the Rosenblatt process $R=(R_t)_{t\in\mathds{R}}$. This process was introduced by Taqqu \cite{taqqu1975weak}, see also \cite{dobrushin1979non} and arises as a limiting process in so-called (functional) non-central limit theorems, analogously to FBM appearing in central limit theorems for correlated random walks. We will not give a formal definition of the Rosenblatt process, an ad hoc definition can be given using an iterated Wiener-It\=o integral, see \cite{taqqu1978representation}. Instead, we restrict ourselves to listing the properties of $R$ which are relevant to verify the local time persistence result. A comprehensive source for all stated facts is Taqqu's survey article \cite{taqqu2011rosenblatt}. Unlike FBM, the Rosenblatt process can only be defined for $H\in(\nicefrac12, 1)$. For any such $H$, $R$ is uniquely defined (up to multiplication by a constant) and satisfies
\begin{itemize}
	\item[$\bullet$] $R$ has H\"older continuous paths a.s. for any H\"older exponent $\gamma< H$,
	\item[$\bullet$] $R$ is $H$-sssi.
\end{itemize}
\begin{theorem}[Persistence of local time for the Rosenblatt process]\label{thm:persistence3}
	Let $R$ denote the $H$-sssi Rosenblatt process, $H\in(\nicefrac12,1)$ and denote by $\ell$ its local time at $0$. Then there is a constant $c_R\in(0,\infty)$ such that, as  $T\to\infty$,
	\[
	\PP(\ell(0,T]\leq 1)\sim c_R T^{-(1-H)}
	\]
\end{theorem}
\begin{proof}
	In principle, we can apply the same arguments as for FBM, but the corresponding preliminary results for the Rosenblatt process needed to verify conditions \ref{eq:continuity} and \ref{eq:singularity} are less well known. We thus give a slightly more explicit version of the argument. Existence of square integrable (in space) local times for $R$ has been shown in \cite{shevchenko2011properties}. Let us show continuity of the cumulative local time process $(\ell_t)_{t\geq 0}$. Geman's sufficient criterion \cite[Therorem B (I)]{geman1976note} for the continuity of the local time can be restated as follows for two-sided stationary increment processes:
	\begin{equation}\label{eq:gemancrit}
	\int_{-1}^1 \sup_{\epsilon>0}\frac{1}{\epsilon}\PP(|R_s|<\epsilon) \di s <\infty.
	\end{equation} 
	To show that \eqref{eq:gemancrit} holds, we use results of Veillette and Taqqu \cite{veillette2013properties}, who studied the distribution of $R_1$ extensively. In particular, they show that $R_1$ has a smooth density and the same holds, by self-similarity, for $R_s, s\in{\mathds{R}\setminus\{0\}}$. Let $f_s$ denote the density of $R_s$. Then \eqref{eq:gemancrit} is satisfied, if
	\[
	g(s):=\limsup_{\epsilon\downarrow 0}\frac{1}{2\epsilon}\int_{-\epsilon}^\epsilon f_s(u)\di u
	\] 
	is integrable around the origin. But by smoothness of $f_s$, we have $g(s)=f_s(0)<\infty$ and thus $g(s)=s^{-H}g(1)$ for any $s>0$. Consequently, $g$ is integrable around $0$ and \eqref{eq:gemancrit} is satisfied.
	
	Turning to \ref{eq:singularity}, we argue using the well known fact, see e.g \cite{kahane1985some}, about H\"older continuity and Hausdorff dimension of the level sets of a real function $f$: If $f$ is $\gamma$-H\"older continuous with $\gamma\in(0,1)$ then the Hausdorff dimension of its level sets is at most $1-\gamma.$ This is sufficient to complete the argument in the same fashion as for FBM. 
\end{proof}

\begin{remark}
	The following general recipe may be used to verify the conditions of Theorem~\ref{thm:persistence1}: If an $H$-sssi \textsf{LT} process has sufficiently high moments, then H\"older-continuity of the paths can always be inferred from the Kolmogorov-Chentsov Continuity Theorem \cite{chentsov1956weak} and \ref{eq:singularity} is then always satisfied. Additionally, if the transition density at $0$ exists, then \ref{eq:continuity} is always satisfied.
\end{remark}

\section{Proof of Theorem {\ref{thm:persistence1}}}\label{sec:proofs}

Let $X$ be as in Theorem~\ref{thm:persistence1}
and let $\ell$ be its local time at $0$. 
Recall that the corresponding additive functional $(\ell_t)_{t\in\mathds{R}}$, is given by
\[
\ell_t=\begin{cases}
\ell(0,t], & \text{ if } t>0 \\
-\ell(t,0], & \text{ if } t\leq 0,  
\end{cases}
\]
and denote its right-continuous inverse by $L=(L_{x})_{x\in\mathds{R}}$. It is straightforward from \eqref{eq:LTpathwise} that $(\ell_t)_{t\in\mathds{R}}$ is $(1-H)$-self-similar, hence $L$ is $\nicefrac{1}{1-H}$ self-similar. By \ref{eq:continuity}, $\ell$ has no atoms and hence $L$ is strictly increasing. By \ref{eq:singularity}, $\mathsf{supp}(\ell)$ is nowhere dense. $L$ is then a monotone pure jump process. Consequently, it induces a locally finite, purely atomic random measure $\hat{\ell}$ on $\mathfrak{B}(\mathds{R})$. We say a random measure is $\beta$-scale-invariant, if its additive functional is a $\beta$-self-similar process. It thus follows from self-similarity of $L$ that $\hat\ell$ is $\nicefrac{1}{1-H}$-scale-invariant. Because $\hat{\ell}$ is purely atomic, we may identify it with a point process on $\mathds{R}\times (0,\infty)$, see Lemma~\ref{lem:measuretopp}. This point process is denoted by $\hat{N}$ and its intensity measure by $\hat{\Lambda}.$ The key observation of our argument is that $\hat{\Lambda}$ is entirely determined (up to a multiplicative constant) by the invariance properties of $\hat{\ell}$.
\begin{proposition}\label{prop:powerlaw}
	Let $\hat{N}$ denote the point process representation of the inverse local time measure $\hat{\ell}$. Then the corresponding intensity measure $\hat{\Lambda}$ is given by
	\begin{equation}\label{eq:powerlawrep}
	\hat{\Lambda} (\di x \times \di m)= cm^{-1-(1-H)}\di x\di m,
	\end{equation}
	for some finite constant $c>0$.
\end{proposition}
We postpone the proof of Proposition~\ref{prop:powerlaw} to the next section, but note that subject to the validity of Proposition~\ref{prop:powerlaw}, all that remains to establish Theorem~\ref{thm:persistence1} is to relate the tail behaviour of $\hat\Lambda$ to the tail behaviour of $\ell$.

\begin{proof}[Proof of {Theorem~\ref{thm:persistence1}}]
	We observe that 
	\[
	\PP(\ell(0,T]\leq 1)= \PP(\ell(0,T]< 1) =\PP(L_1>T),\quad T>0,
	\]
	i.e. we obtain lower tail bounds for $\ell_t$ from upper tail bounds for $L_1$. Let $\hat{N}$ denote the point process representation of $\hat{\ell}$ and note that $L_1=\int_0^1\int_0^\infty m \hat{N}(\di x \times \di m)$.  Fix any $r>0$. Since $\hat{N}$ is a simple point process and $\hat{\ell}$ is purely atomic, we have by standard results from random measure theory, e.g. \cite[Prop. 9.1.III(v)]{DVJII},
	\begin{equation}\label{eq:ineq1}
	\hat{N}([0,1]\times(r,\infty))=\lim_{n\to\infty}\sum_{k=1}^{n}\mathds{1}\left\{\hat{\ell}\left(\frac{k-1}{n},\frac{k}{n}\right]>r\right\}, \text{ a.s}.
	\end{equation}
	Set \[
	P_{k,n}:=\PP\left(\hat{\ell}\left(\frac{k-1}{n},\frac{k}{n}\right] > r\right), \quad 1\leq k\leq n, n=1,2,\dots,
	\]
	taking expectations in \eqref{eq:ineq1}, we obtain
	\[
	\limsup_{n\to\infty}\sum_{k=1}^n P_{k,n}\leq \EE \hat{N}([0,1]\times(r,\infty))\leq	\liminf_{n\to\infty}\sum_{k=1}^nP_{k,n},
	\]
	having applied Fatou's Lemma and the inverse Fatou's Lemma, i.e.\[\sum_{k=1}^n P_{k,n}\overset{n\to\infty}{\longrightarrow} \EE \hat{N}([0,1]\times(r,\infty)).\] For any $\delta\in (0,1)$ we may thus fix $1\leq N_\delta<\infty$ such that
	\begin{equation}\label{eq:sandwich}
	(1-\delta)\int_{r}^\infty c m^{-1-(1-H)}\di m \leq \sum_{k=1}^{N_\delta} P_{k,N_{\delta}} \leq (1+\delta)\int_{r}^\infty c m^{-1-(1-H)}\di m, 	
	\end{equation}
	where we have used that \[\EE \hat{N}([0,1]\times(r,\infty))= \hat{\Lambda}([0,1]\times(r,\infty))= \int_{r}^\infty c m^{-1-(1-H)}\di m, \]
	according to Proposition~\ref{prop:powerlaw}. Using that $P_{k,N_{\delta}}=\PP(L_{\nicefrac{1}{N_\delta}}>r)$ by the stationarity of the increments of $L$, we can rewrite \eqref{eq:sandwich} as
	\begin{equation*}
	(1-\delta)\frac{c}{1-H}r^{-(1-H)} \leq N_{\delta}\PP(L_{\nicefrac{1}{N_\delta}}>r) \leq (1+\delta)\frac{c}{1-H}r^{-(1-H)}, 	
	\end{equation*}
	and applying the $\nicefrac{1}{1-H}$-self-similarity of $L$ and rearranging terms yields
	\begin{equation*}
	(1-\delta)\frac{c}{1-H}\left({r}{N_\delta ^{\nicefrac{1}{1-H}}}\right)^{-(1-H)} \leq \PP\left(L_1>rN_{\delta}^{\nicefrac{1}{1-H}}\right) \leq (1+\delta)\frac{c}{1-H}\left({r}{N_\delta ^{\nicefrac{1}{1-H}}}\right)^{-(1-H)}, 	
	\end{equation*}
	i.e.
	\[
	\PP(L_1>T)= \frac{c}{1-H}(1+o(1))T^{-(1-H)}, \quad \text{ as }T\to\infty,
	\]
	and Theorem~\ref{thm:persistence1} is proved, subject to Proposition~\ref{prop:powerlaw}.
\end{proof}


\section{Auxiliary results}\label{sec:auxiliaryproof}
We now provide the background needed to complete the proof of Theorem~\ref{thm:persistence1}, starting with some basics of random measure theory, namely we introduce Palm distributions and discuss some of their key properties and state a few results about scale invariant random measures.

\subsection{Palm distributions and mass-stationarity}
Here and in the following subsection, we take a general point of view on $\ell$ and its distributional properties as a random measure, in particular we can forget about the process $X$ from which it is derived. We instead consider some complete measurable space $(\Omega,\mathfrak{F})$, equipped with a $\sigma$-finite measure $Q$. We denote by $E_Q(\cdot)$ integration with respect to $Q$. A measurable map $\xi$ from $\Omega$ into the space $(\mathcal{M},\mathfrak{B}(\mathcal{M}))$ of locally finite measures on $\mathds{R}$, equipped with its Borel-$\sigma$-field is called a \emph{random measure}, even though we stress that $Q$ need not be a probability distribution. We call $Q$ a \emph{quasi-distribution}, to distinguish it from the measures which are elements of $\mathcal{M}$ and set $Q_\xi=Q\circ\xi^{-1},$ i.e.
\[
Q_{\xi}(G)=Q(\xi\in G), \quad G\in\mathfrak{B}(\mathcal{M}).
\]
The intensity measure of $\xi$ under $Q$ (or $Q_\xi$) is given by
\[
\Lambda_\xi (A):=E_{Q}\xi(A)=\int \nu(A)Q_{\xi}(\di \nu), \; A\in\mathfrak{B}(\mathds{R}).
\]
In what follows, we frequently consider $Q$ directly as a quasi-distribution on $\mathfrak{B}(\mathcal{M})$ without explicit reference to a (canonical) random measure $\xi$ with distribution $Q$, consequently we denote the associated intensity measure by $\Lambda_Q$.
Whenever we discuss a probability measure, then we indicate this by using blackboard-face symbols, e.g. $\mathds{P},\mathds{P}_\xi,\mathds{E}_\xi,$ etc. Futhermore, to formalise our discussion of stationarity properties, we use the shift group $(\theta_t)_{t\in\mathds{R}}$ on $\mathds{R}$. Note that the $\theta_t, t\in\mathds{R},$ act measurably on $(\mathcal{M},\mathfrak{B}(\mathcal{M})),$ and in particular we have that
\[
\theta_{-t}\nu(A)=\nu(A+t)=\nu\circ\theta_t(A), \quad A \in\mathfrak{B}(\mathds{R}),
\]
where $A+t:=\{a+t, a\in A\}$. A quasi-distribution $Q$ on $\mathfrak{B}(\mathcal{M})$ is invariant under the shifts $(\theta_t)_{t\in\mathds{R}}$, i.e. \emph{stationary}, if
\[
Q(G)=Q(\{ \theta_t\nu, \nu\in G \}), \; t\in\mathds{R}.
\]
If $Q$ is stationary and satisfies \[\lambda_Q:= \Lambda_Q\left((0,1]\right)\in (0,\infty)\] then it follows immediately that $\Lambda_Q(\di s)=\lambda_Q \di s$, i.e. $\Lambda_Q$ is a constant multiple of Lebesgue measure. We call $\lambda_Q\in[0,\infty]$ the \emph{intensity} of $Q$.
Note that the quasi-distribution $Q$ is a place holder for a stationarised version of the distribution of the local time measure $\ell$, the corresponding construction is discussed below. At the moment it is more beneficial to stay in the general setting. However, the following assumptions on the support of $Q$
\[\mathcal{S}_{Q}:=\mathsf{supp}(Q)=\mathcal{M}\setminus \bigcup_{N\in\mathfrak{B}(\mathcal{M}): Q(N)=0} N \]
are justified in view of our applications. Let $o$ denote the $0$-measure, then 
\[ o \notin \mathcal{S}_{Q},\text{ i.e. $Q$ is \emph{non-zero},}\]
and
\begin{equation} \label{eq:RCstar}\tag*{\textsf{R}}
\{\nu_t, t\in\mathds{R}\}=\mathds{R}, \; \text{ for all } \nu\in \mathcal{S}_{Q}.
\end{equation} 	
We have $\lambda_Q>0$, if $Q$ is non-zero and stationary. Note that \ref{eq:RCstar} is in congruence with condition \ref{eq:continuity}, the details are given in Lemma~\ref{lem:equivalentconditions}.

We now turn to the subject of Palm distributions of a stationary, non-zero quasi-distribution $Q$. Fix any $A\in\mathfrak{B}(\mathds{R})$ with finite and positive Lebesgue measure, then the quasi-distribution defined by
\[
P_Q(G)=\frac{1}{\int_A \di s }\int \int_A \mathds{1}\{\theta_{-t}\nu\in G \} \nu (\di t) Q(\di \nu), \quad G\in \mathfrak{B}(\mathcal{M}),
\]
is independent of the choice of $A$, see Lemma~\ref{lem:PDA}. It is referred to as the \emph{Palm measure} of $Q$. If $Q$ has finite intensity $\lambda_Q$ then a probability distribution is defined by
\[
\PP_Q(G)= \frac{P_Q (G)}{\lambda_Q},\quad G\in \mathfrak{B}(\mathcal{M}),
\]
and is called the \emph{Palm distribution} of $Q$. Conversely, we say that a random measure $\xi$ is \emph{Palm-distributed}, if its distribution $\PP_{\xi}$ is the Palm distribution of some stationary quasi-distribution $Q$. Similarly, a random measure is Palm distributed if its distribution is the Palm distribution of some $Q$. It is well known, see e.g. \cite[Lemma 3.3]{Z1}, that the almost sure properties of $Q$ and $\PP_Q$ agree (up to shifts). We may thus assume that any $\nu\in \mathcal{S}(\PP_Q)$ has the property indicated in \ref{eq:RCstar}. The latter entails that the right continuous inverse $(\nu^{-1}_x)_{x\in\mathds{R}}$ of the additive functional $(\nu_t)_{t\in\mathds{R}}$ of $\nu$ is strictly increasing and that we have \[\nu_{\nu^{-1}_x}=x,\quad \text{ for all } x\in\mathds{R}.\] This allows us to define the measurable group of \emph{random time shifts}
\[
\hat{\theta}_x:=\hat{\theta}_x(\nu):=\theta_{\nu^{-1}_x},\; x\in\mathds{R}.
\]
Note that we may pick $A=[0,1]$ in the definition of $P_Q$ and change variables according to the random time change to obtain the alternative representation
\begin{equation}\label{eq:tcPalm}
\PP_Q(G)=\frac{1}{\lambda_Q}\int \int_0^{\nu_1}\mathds{1}\{ \hat{\theta}_{-x}\nu\in G\} \di x Q(\di \nu ), \; G\in\mathfrak{B}(\mathcal{M}).
\end{equation}
Let us very briefly discuss the intuition behind Palm distributions in general and the use of the shifts $\hat{\theta}_x$ in particular. Palm distributions originated in queuing theory \cite{palm1943intensitatsschwankungen}. The concept is easiest understood for simple stationary point processes, which corresponds to $Q$ being the distribution of a random counting measures in our setting. In this case, the Palm distribution may be interpreted as a description of the distribution of the point process seen from a `typical point', an intuition which can be made precise using ergodic theory, see e.g. the discussion in \cite[Chapter 8]{thorisson2000coupling}. At the heart of Palm theory lies a duality principle \cite{thorisson1995time}, which can be paraphrased as\\

\noindent``A point process is stationary, if and only if its Palm version is stationary under point-shifts.''\\

\noindent Since we are not working with point processes, point-stationarity needs to be replaced by `mass-stationarity', which corresponds to invariance under the shifts $\hat{\theta}_x$. Also, for our purpose, the backward implication in this statement is not needed. We only rely on the observation, that a Palm distributed random measure is stationary w.r.t.\ to intrinsic shifts, i.e. shifts by mass points of its realisation. The following characterisation theorem due to Last, M\"orters and Thorisson provides a general formulation of this principle.
\begin{theorem}[{\cite[Theorem 3.1]{LMT14}}]\label{thm:LMT}
	Assume that $\nu$ is a diffuse random measure on a probability space $(\Omega,\mathfrak{F},\PP)$ with $\PP(\nu(-\infty,0)<\infty)=\PP(\nu(0,\infty)<\infty)=0.$ Then the following statements are equivalent
	\begin{enumerate}[(i)]
		\item $\nu$ is Palm-distributed;
		\item $\nu$ is mass-stationary, i.e.\ for all bounded Borel-sets $C$ such that $$\int_C \di x>0\quad\text{and}\quad\int_{\partial C} \di x=0$$ and all measurable functions $f:\Omega\times \mathds{R}\to [0,\infty)$
		\[
		\EE_{\PP}\int\int \1_C(u)\frac{\1_{C-u}(s)}{\nu(C-u)}f(\theta_s,s+u)\nu(\di s)\di u = \EE_{\PP}\int \1_C(u) f(\theta_0,u)\di u;
		\]
		\item $\PP\circ \hat{\theta}_x^{-1} = \PP$ for all $x\in \mathds{R}$.
	\end{enumerate} 	
\end{theorem}
For the local time $\ell$, the corresponding stationary quasi-distribution $Q_\ell$ hidden in statement (i) of Theorem \ref{thm:LMT} has the following interpretation: Let us set \[
Q_{\ell}(G)=\EE \int \mathds{1}\{\ell^y(\cdot)\in G  \}\di y, \quad G\in\mathfrak{B}(\mathcal{M}), o\notin G,
\]
recalling that $o$ is the trivial measure. We can think of $Q_{\ell}(G)$ as the `local time at the origin' of a trajectory in the \emph{flow} of $X$, i.e. the quasi-distribution $Q_X$ on trajectories obtained via
\[
Q_X({H})=\int \mathds{1}\{ X+y\in{H}   \} \di y,\; {H}\in \mathfrak{Z}(\mathcal{C}),
\]
where $\mathcal{C}$ is a suitable path space equipped with the $\sigma$-field $\mathfrak{Z}(\cdot)$ of cylinder sets. $Q_X$ can be thought of as mixture of the law of $X$ w.r.t. Lebesgue measure in the `origin' $X(0)$ to obtain a stationary measure on paths. From $Q_X$ we can derive a stationary version of the occupation measure and then disintegrate to obtain $Q_\ell$. Note however, that $Q_{\ell}$ and $\PP_{\ell}$ need not be interpreted in this way -- it is only necessary that the \emph{distribution} of local time as a random measure is a Palm distribution.\\ 

In fact, we do not require the original definition (ii) of mass-stationary but only its characterisation through random shifts (iii). The crucial result for our approach is the implication (i)$\Rightarrow$(iii).
\begin{lemma}\label{lem:dual}
	Let $Q$ be a stationary non-zero measure on $\mathfrak{B}(\mathcal{M})$ with $\lambda_Q<\infty$, then its Palm distribution $\PP_Q$ is stationary with respect to the random shifts $(\hat{\theta}_x)_{x\in\mathds{R}}$.
\end{lemma}
In the appendix, we provide a short, direct proof of Lemma \ref{lem:dual} without reliance on Theorem \ref{thm:LMT}. To actually apply Lemma\ref{lem:dual}, we need that $\ell$ is indeed Palm distributed, which is well-known. Before we formulate the corresponding statement as Proposition \ref{prop:LTisPALM} below, we briefly discuss scale-invariance of the local time which reflects the self-similar nature of the underlying process $X$.

\subsection{Bi-scale-invariance}\label{sec:stat}
So far, we have focussed our discussion of random measures on invariance with respect to time shifts only. Now we additionally consider scale-invariance of measures, which is the counterpart to self-similarity of processes. The essential observation of this section is that stationarity combined with scale-invariance of a quasi-distribution or Palm distribution determines the corresponding intensity measures entirely up to a multiplicative constant.

We illustrate this by means of marked point processes on the real line. We consider $(\Omega,\mathfrak{F},\PP)$, i.e. we work under a probability measure. An \emph{extended marked point process with positive marks (EMPP)} is a point process on $\mathds{R}\times (0,\infty)$
which is a.s. finite on all sets of the form $A\times M$ for bounded $A\in \mathfrak{B}(\mathds{R})$
and Borel sets
\[
M\subset \left(\epsilon,\frac{1}{\epsilon}\right), \; \text{ for some } \epsilon>0.
\]

Fix $\beta>0$ and $r\in(0,1)$. We define a rescaled point process by \[S^\beta_r N (A\times M):= N(rA \times r^\beta M),\]
where $c M:=\{cm, m\in M\}$ for any $c\in\mathds{R}\setminus\{0\}$, hence a contraction by factor $r$ in the time domain is combined with a contraction by $r^\beta$ in the mark space into the operator $S_r^\beta$. An EMPP $N$ on $\mathds{R}\times (0,\infty)$ is called \emph{$\beta$-bi-scale-invariant} if, for any $r\in(0,1)$,
\[
\PP\circ (S^\beta_rN)^{-1} = \PP\circ N^{-1}.
\]
Similarly, the EMPP is stationary, if its distribution is invariant under the shifts $(\theta_t)_{t\in\mathds{R}}$ applied to the time domain only. We now recall a well known result about point processes: if the intensity measure of the EMPP is finite, then it follows from bi-scale invariance together with stationarity that the intensity measure must be a product of a multiple of Lebesgue measure in time and a hyperbolic law on the marks. That this is the case can be seen by noting that stationarity implies homogeneity in time of the intensity measure, i.e. it must be a multiple of Lebesgue measure.  Additionally, bi-scale-invariance is transferred into stationarity on the mark space when mapped to logarithmic coordinates, hence the intensity on the marks is a multiple of Lebesgue mesasure in logarithmic coordinates which is translated into a hyperbolic law when reversing the coordinate transform.
\begin{proposition}\label{prop:scaling}
	Let $N$ be a bi-scale invariant, stationary, extended\footnote{Instead of using the notion of an extended point process, one can also consider a classical locally finite point process if one equips the closure of the mark space with a metric which places $0$ infinitely far away from any positive mark.} marked point process on $\mathds{R}\times(0,\infty)$ with positive and locally finite intensity measure $\Lambda_N$. Then $\Lambda_N$ necessarily is of the form 
	\begin{equation}\label{eq:intensityscaling}
	\Lambda_N (\di t \times \di m)= cm^{-1-\nicefrac{1}{\beta}}\di t\di m.
	\end{equation}
\end{proposition}
\begin{proof}
	We only give an outline of the precise argument. A similar derivation in more detail can be found in \cite[Chapter 12]{DVJII}, for the case of an extended marked Poisson process. Let $\Lambda_N$ denote the intensity measure of $N$. We assume that \[
	\Lambda_N(\{0\}\times (0,\infty))=0,
	\]
	which holds if $N$ has almost surely no points at $0$ and that $\Lambda_N$ is absolutely continuous with respect to $2$-dimensional Lebesgue measure. We use a logarithmic change of coordinates, which makes it necessary to decompose $N$ into a marked point process $N_+$ on $(0,\infty)$ and a marked point process $N_-$ on $(-\infty,0)$ with associated intensity measures $\Lambda_\pm$. 
	
	Let us first consider $\Lambda_+$ only. By the logarithmic change of coordinates \[(t,m)\mapsto (\log t,\beta\log t-\log m),\quad t\in\mathds{R}, m\in(0,\infty),\] bi-scale-invariance is turned into shift-invariance and consequently under the new coordinates, $\Lambda_+$ must be a product of Lebesgue measure and some absolutely continuous measure $\rho_+$ on the (coordinate transformed) mark space, whenever the intensity in time is finite. In particular, reversing the coordinate transform, we obtain \[
	\Lambda_+(\di t\times \di m) =\frac{ \phi_+({t^\beta}/{m})}{tm}\di t\di m
	\]
	for some locally integrable density $\phi_+$ of $\rho_+$ on $(0,\infty)$, see \cite[p. 258]{DVJII}\footnote{As noted above, the result there is for a Poisson process, however the statement transfers immediately to the general case. This is also discussed in \cite{DVJII} on pp. 260-262, however not in as much detail as the Poisson case.
	}. A similar representation holds for $\Lambda_{-}$ with a density $\phi_{-}$. The additional assumption of stationarity in the time domain now implies that we must have
	\[
	\phi_+(m)=\phi_{-}(m)=cm^{-\nicefrac{1}{\beta}},
	\]
	for some $c>0$ and thus 
	\eqref{eq:intensityscaling} must be satisfied.
\end{proof}
To apply this representation to random measures we recall that atomic random measures can be bijectively mapped to EMPPs. Let $\mathcal{N}$ be the space of locally finite extended marked point processes with positive marks $N$ on $(\mathds{R}\times(0,\infty))$ satisfying
\begin{equation}\label{eq:integrabilitycondition}
\int (m\wedge 1) \Lambda_N (A\times \di m)<\infty
\end{equation}
for any bounded Borel set $A$, and let $\mathcal{M}_a\subset\mathcal{M}$ denote the locally finite, purely atomic random measures on $\mathfrak{B}(\mathds{R})$.
\begin{lemma}[{\cite[Lemma 9.1.VII]{DVJII}}]\label{lem:measuretopp}
	There is a bijection mapping $\mathcal{N}$ onto $\mathcal{M}_a$.
\end{lemma}
In principle, the bijection of Lemma~\ref{lem:measuretopp} just consists of interpreting, for given $\xi\in \mathcal{M}_a$, a point $x\in\mathsf{supp}(\xi)$ with $\xi(\{x\})=m>0$ as a pair $(x,m)\in \mathds{R}\times(0,\infty) $ and the collection of all such points forms a marked point process $N_{\xi}$. To deal with accumulation points of $\mathsf{supp}(\xi)$, one needs to consider extended MPPs. Note that if the intensity measure $\Lambda_\xi$  of some random measure $\xi$ under this bijection assigns infinite mass to a bounded open interval $A$, then by  local finiteness of $\xi$ this must be the consequence of infinitely many smaller and smaller atoms. Thus $N_{\xi}$ puts finite mass on any open set $A\times (\delta,\infty)$ for $\delta>0$ and is thus locally finite on $(0,\infty)$ in the extended sense. Since we wish to apply Lemma~\ref{lem:measuretopp} to measures without fixed atoms, we observe that the bijection is preserved when restricted to the subspace of measures without fixed points and the subspace of without fixed points, respectively.

The version of $\beta$-bi-scale-invariance for quasi-distributions is just called $\beta$-scale-invariance, as defined Section~\ref{sec:proofs} for probability distributions. Let us quickly recall this definition in the notation of this section. Let $\mathds{P}$ be a probability distribution on $\mathfrak{B}(\mathcal{M})$. $\PP$ is called $\beta$-{scale-invariant}, if for any $r\in(0,1)$, $B\in\mathfrak{B}(\mathds{R})$  we have 
\[
\PP (G) = \PP \left(\{R^{\beta}_r\nu, \nu\in G \}\right),\; G\in \mathfrak{B}(\mathcal{M})
\]
where
\[ \left(R^\beta_r \nu\right) (A):= r^{-\beta}\nu(rA),\] or short $\PP\circ (R^\beta_r)^{-1}=\PP$. When considering a non-finite quasi-distribution, one needs to add an additional factor rescaling the total mass: A stationary non-finite quasi-distribution $Q$ with finite intensity $\lambda_Q$ is called $\beta$-scale-invariant if \[
Q\circ \big(R^\beta_r\big)^{-1}=r^{\beta-1} Q.
\]
That this is the correct notion of scale-invariance for quasi-distributions can be seen by looking at their Palm distributions:
\begin{lemma}[{\cite[Statement 2.3]{Z1}}]
	A stationary, non-zero, non-finite quasi distribution $Q$ is $\beta$-scale-invariant if and only if its Palm distribution $\PP_Q$ is $\beta$-scale-invariant.
\end{lemma}

\subsection{Proof of Proposition {\ref{prop:powerlaw}}}
The last ingredient needed to prove Proposition {\ref{prop:powerlaw}} is the following result which establishes that $\ell$ is indeed distributed according to a scale-invariant Palm-distribution. 
\begin{proposition}[{\cite[Proposition 6.9.]{Z3}}]\label{prop:LTisPALM}
	If $X$ is $H$-sssi and \textsf{LT}, then $\ell$ is an $(1-H)$-scale-invariant, Palm-distributed random measure.
\end{proposition}
We may rephrase the statement of Proposition~\ref{prop:LTisPALM} as a statement about $\PP_{\ell}$, the distribution of the local time as a random measure: There exists some quasi-distribution $Q_{\ell}$, such that $\PP_{\ell}$ is the Palm-distribution of $Q_{\ell}$.\\

Finally, we are in the position to prove {Proposition~\ref{prop:powerlaw}} and thus conclude the proof of Theorem~\ref{thm:persistence1}.
\begin{proof}[Proof of {Proposition~\ref{prop:powerlaw}}]
	The representation \eqref{eq:powerlawrep} follows immediately from Proposition~\ref{prop:scaling} upon showing $\nicefrac{1}{1-H}$-bi-scale-invariance of $\hat{N}$. This is, in turn, equivalent to $\nicefrac{1}{1-H}$-scale-invariance and stationarity of $\hat{\ell}$. The scale-invariance has already been established in the opening paragraph of Section~\ref{sec:proofs}. We now show that $L$ has stationary increments and thus $\hat{\ell}$ is a stationary random measure. By Proposition~\ref{prop:LTisPALM}, $\ell$ is Palm-distributed. For $x\in \mathds{R}$ set $t(x):=\inf\{t: \ell(t)>x\}=L_x$. Fix $x_0\in \mathds{R}$ and consider a finite family of points $x_1,\dots,x_n$ with $x_0<x_1<\dots <x_n$ and the corresponding random times $t(x_i), i=1,\dots,n.$ Almost surely, $\{t(x_i), i=1,\dots,n\}\subset\mathsf{supp}(\ell)$ and because $\ell$ is Palm-distributed, we may apply Lemma~\ref{lem:dual}, to obtain \[
	\big(t(x_i)-t(x_0)\big)_{i=1}^n\overset{d}{=} \big(t(x_i-x_0)-t(0)\big)_{i=1}^n=\big(L_{x_i-x_0}\big)_{i=1}^n
	\]
	and thus $L$ has stationary increments. Consequently, $\hat{\ell}$ is $\nicefrac{1}{1-H}$-scale-invariant and stationary and this concludes the proof of the Proposition~\ref{prop:powerlaw}.
\end{proof}

\section{Appendix}\label{sec:auxproofs}

\subsection{Well-definedness of local time}

As mentioned in Section \ref{sec:results1}, it cannot be a priori excluded that $0\in\mathcal{R}$ and thus $\ell=\ell^0=\ell^{X(0)}$ is not defined. Following \cite{Z3}, this can be avoided by working with the averaged version
\[
\tilde{\PP}_{\ell}(\cdot)={\int_{0}^1 \PP_{\ell,t}(\cdot)\di t},
\]
where $\PP_{\ell,t}(\cdot)$ denotes the distribution of $\theta_{-t}\ell^{X_t}, t\in\mathds{R}$. This trick can always be used in the stationary increment case, but it has the drawback that this version of the local time does not have the usual path-wise interpretation. 
It is, however, easily seen that $\tilde{\PP}_{\ell}$ and $\PP_{\ell}$ generate the same distribution of random measures, if the latter is defined, see the discussion in \cite[p. 132]{Z3}. We remark that the use of $\tilde{\PP}_{\ell}$ is obsolete, if $X$ a.s. has continuous paths. 

\subsection{Alternative proof of {Lemma \ref{lem:dual}}}
The calculation is adapted from \cite[Lemma 2.3]{miyazawa1994modified}, see also \cite[Theorem 3.1]{miyazawa1900palm}.
\begin{proof}[Proof of {Lemma~\ref{lem:dual}}]
	Let $G\in\mathfrak{B}(\mathcal{M})$ and fix any $r>0$. Then the stationarity of $Q$ implies that, for any $y>0$,
	\begin{align*}
	\int \int^{\nu(0,r]+y}_{\nu(0,r]} \mathds{1}_G\circ \hat{\theta}_x \di x Q(\di \nu) \,&= \, \int \int^{\nu(0,r]+y}_{\nu(0,r]} \mathds{1}_G\circ \hat{\theta}_x \di x Q(\di \nu \circ \theta_{-r})\\
	&= \, \int \int^{-\nu(r,0]+y}_{-\nu(r,0]} \mathds{1}_G\circ \hat{\theta}_{x-\nu(-r,0]} \di x Q(\di \nu )\\
	&= \, \int \int^{y}_{0} \mathds{1}_G\circ \hat{\theta}_{x} \di x Q(\di \nu ),
	\end{align*}
	and an analogous calculation can be made for $r<0$. Thus we have that, using \eqref{eq:tcPalm},
	\begin{align*}
	\PP_Q(\hat{\theta}_x^{-1}G)\,&=\,\frac{1}{\lambda_Q}\int \int_{0}^{\nu_1} \mathds{1}_G\circ \hat{\theta}_{z+x}\di z  Q(\di \nu)\\
	&=\,\frac{1}{\lambda_Q}\left(\int \int_{0}^{\nu_1} \mathds{1}_G\circ \hat{\theta}_{z}\di z  Q(\di \nu)-\int \int_{0}^{x} \mathds{1}_G\circ \hat{\theta}_{z}\di z  Q(\di \nu)\right.\\
	&\phantom{=\;}+\left.\int \int_{\nu_1}^{\nu_1+x} \mathds{1}_G\circ \hat{\theta}_{z}\di z  Q(\di \nu)\right)\\
	&=\,\frac{1}{\lambda_Q}\int \int_{0}^{\nu_1} \mathds{1}_G\circ \hat{\theta}_{z}\di z  Q(\di \nu) \,=\, \PP_Q(G).
	\end{align*}
\end{proof}

\subsection{Further calculations}

\begin{lemma}\label{lem:equivalentconditions}
	Let $\ell$ be a non-zero local time of an $H$-self-similar process satisfying \ref{eq:continuity}, then 
	\[
	\{ \ell_t, t\in \mathds{R}  \} =\mathds{R}.
	\] 
\end{lemma}
\begin{proof}
	From the path-wise representation of $\ell$, it follows immediately that $(\ell_t)_{t\geq 0}$ is an $1-H$-self-similar process. Assume that $\PP(\ell_t\in [0,K] \text{ for all } t)>\delta>0$ for some $K<\infty,\delta>0$. Fix $b>0$, then we also have $\PP(\ell_{Tb}\in [0,K] )>\delta$ for any $T$ and by self-similarity, this means $\PP(\ell_b\in [0,KT^{-(1-H)}])>\delta$ for any $T>0$. Thus $\ell(0,b]=0$ with probability exceeding $\delta$. Since $b$ was arbitrary, this means $\ell_t\equiv 0$ on $[0,\infty)$. A similar argument works for negative $b$ and we obtain a contradiction to $\ell$ being non-zero. Hence the range of $(\ell_t)_{t\geq 0}$ is a.s. not bounded on either side. By \ref{eq:continuity}, i.e. continuity of $\ell_t$, the range must cover all of $\mathds{R}$.
\end{proof}
\begin{lemma}\label{lem:PDA} 
	The definition of the Palm measure $P_Q$ of $Q$ does not depend on the choice of the set $A$. 
\end{lemma}
\begin{proof} For probability distributions this is shown in \cite[Lemma 11.2]{kallenberg2006foundations}. The proof transfers easily to quasi-distributions. Let $Q$ be stationary and non-zero. Fix $G\in\mathfrak{B}(\mathcal{M})$, let $g(\nu)=\mathds{1}\{\nu\in G\}$ and consider the measures
	\[
	\nu_g(B)=\int_B g(\theta_{-s}\nu)\nu(\di s),\; B\in\mathfrak{B}(\mathds{R}),
	\]
	then the definition of the Palm measure reads
	
	\begin{equation}\label{eq:PMFrac}
	P_Q(G)=\frac{\int \nu_g(A) Q(\di \nu)}{\int_A \di s}.
	\end{equation}
	We claim that $\theta_{-t}\nu_g = (\theta_{-t}\nu)_g$, then we have for any measurable function $h$,
	\[
	\int h\left(\theta_{-t}\nu_g\right)\di Q = \int h\left((\theta_{-t}\nu)_g\right)\di Q =\int h(\nu_g)\di Q
	\]	
	by stationarity of $Q$. This means that stationarity of $Q$ is preserved under the operation $(\cdot)_g$, hence the numerator in \eqref{eq:PMFrac} does not change under translations of $A$, i.e. is a multiple of Lebesgue measure. To prove the claim we note that, for any Borel set $B$,
	\begin{equation*}
	\begin{aligned}
	\theta_{-t}\nu_g(B)&=\nu_g(B+t)=\int_{B+t}g(\theta_{-s}\nu)\nu(\di s)\\
	&=\int \mathds{1}\{s-t\in B \}g(\theta_{-s}\nu)\nu(\di s)=\int \mathds{1}\{u\in B \}g(\theta_{-u-t}\nu)\nu(\di u + t)\\
	&=\int_{B}g(\theta_{-u}\, \theta_{-t}\nu) \theta_{-t}\nu(\di u) = (\theta_{-t}\nu)_g(B).
	\end{aligned}
	\end{equation*}
\end{proof}

\begin{acknowledgements}
I thank Frank Aurzada for introducing me the subject of persistence probabilities and many discussions on the subject of this article its first draft. I further would like to thank G\"unter Last for providing insightful feedback on an earlier version of the paper and Safari Mukeru for an interesting exchange on the subject of inverse local times of fractional Brownian motion.
\end{acknowledgements}

%
%



\end{document}